\documentclass[11pt,a4paper,reqno]{amsart}

\usepackage{amsmath}
\usepackage{amssymb}
\usepackage{amsthm}
\usepackage[shortlabels]{enumitem}
\usepackage[british]{babel}
\usepackage{hyperref}

\calclayout
\hypersetup{
 colorlinks=true,
 linkcolor=blue,
 citecolor=blue,
 urlcolor=blue,
}

\theoremstyle{plain}
\newtheorem{theorem}{Theorem}[section]
\newtheorem{proposition}[theorem]{Proposition}
\newtheorem{lemma}[theorem]{Lemma}
\newtheorem{corollary}[theorem]{Corollary}
\theoremstyle{remark}
\newtheorem{remark}[theorem]{Remark}
\numberwithin{equation}{section}

\newcommand{\R}{\mathbb{R}}
\newcommand{\CC}{\mathbb{C}}
\newcommand{\Leb}{\mathrm{L}}
\newcommand{\C}{\mathrm{C}}
\newcommand{\Sob}{\mathrm{H}}

\newcommand{\fra}{\mathfrak{a}}
\newcommand{\frb}{\mathfrak{b}}
\newcommand{\frB}{\mathfrak{B}}
\newcommand{\Id}{\mathrm{Id}}
\newcommand{\eps}{\varepsilon}
\newcommand{\abs}[1]{\left\lvert#1\right\rvert}
\newcommand{\norm}[1]{\left\|#1\right\|}
\newcommand{\dd}{\,\mathrm{d}}
\renewcommand{\Subset}{\subset\!\subset}
\DeclareMathOperator{\supp}{supp}

\title[Lions' Maximal Regularity Problem for Differential Operators]
{Lions' Maximal Regularity Problem for Divergence-Form Differential Operators:\\ Failure at the $\frac{1}{2}$-H\"older Endpoint}

\author{Lukas Niebel}
\address[Lukas Niebel]{ETH Z\"urich, Department of Mathematics, R\"amistrasse 101, 8092 Z\"urich, Switzerland.}
\email{lukas.niebel@math.ethz.ch}
\date{21st August 2026}


\begin{document}

\begin{abstract}
  In this work we give a counterexample to maximal $\mathrm{L}^2$-regularity in Lions' problem for divergence-form differential operators. On a bounded interval, we construct a bounded, uniformly elliptic, real scalar diffusion coefficient that is $\frac{1}{2}$-H\"older continuous in time with values in spatial $\mathrm{L}^\infty$. It can be chosen arbitrarily close to the constant coefficient of the heat equation. For zero initial data and a forcing term that is continuous in time with square-integrable spatial values, the unique Lions variational solution has a time derivative that is not square integrable in space-time. Thus $\frac{1}{2}$-H\"older continuity alone does not imply maximal $\mathrm{L}^2$-regularity, even for arbitrarily small scalar perturbations of the heat equation. The construction is based on a lacunary family of oscillatory trigonometric modes localised on shrinking time intervals. The spatial profile and the oscillatory modes, together with their first spatial derivatives, vanish at both endpoints. This permits zero extension of the counterexample to the real line. Tensorisation and localisation by parabolic rescaling then yield real symmetric isotropic counterexamples on $\mathbb{R}^d$ and on every bounded domain $\Omega\subset\mathbb{R}^d$, for all $d\ge1$.
\end{abstract}

\maketitle

\section{Introduction}

\subsection*{Lions' maximal-regularity problem}

Let $T>0$. We consider the non-autonomous Dirichlet problem
\begin{equation}\label{eq:intro-equation}
  \begin{cases}
    \partial_tu(t)-\partial_x\bigl(\fra(t,\cdot)\partial_xu(t)\bigr)=f(t)
     & \text{in }(0,T)\times(0,\pi), \\
    u(t,0)=u(t,\pi)=0
     & \text{for }t\in(0,T),         \\
    u(0)=u_0.
  \end{cases}
\end{equation}
Suppose that $\fra$ is measurable and real-valued and that there
are constants $0<\lambda\le\Lambda<\infty$ such that $\lambda\le \fra(t,x)\le\Lambda$
for almost every $(t,x)\in(0,T)\times(0,\pi)$.
In this setting Lions' variational theory gives for every
\[
  f\in\Leb^2(0,T;\Sob^{-1}(0,\pi)),
  \qquad u_0\in\Leb^2(0,\pi),
\]
a unique variational solution
\[
  u\in\Leb^2(0,T;\Sob_0^1(0,\pi))
  \cap\Sob^1(0,T;\Sob^{-1}(0,\pi)).
\]
Such solutions are called energy solutions in some of the literature;
throughout, we use the term variational solution.
Hence, \eqref{eq:intro-equation} is well posed in this variational class.
However, a priori both the time derivative and the elliptic part take values
only in the negative Sobolev space $\Sob^{-1}(0,\pi)$.

Let us from now on assume that $u_0=0$ and that the forcing has the better regularity
$f\in\Leb^2(0,T;\Leb^2(0,\pi))$. The maximal $\Leb^2$-regularity problem is whether
the equation improves the regularity to
\begin{equation}
  \partial_tu\in\Leb^2(0,T;\Leb^2(0,\pi)),
  \qquad
  -\partial_x(\fra\partial_xu)
  \in\Leb^2(0,T;\Leb^2(0,\pi)).
\end{equation}
A positive answer implies that the variational solution has maximal
$\Leb^2$-regularity in $H$, rather than merely having the regularity defining
the variational class above.

We recall the standard abstract formulation. Throughout, complex Hilbert-space
inner products and sesquilinear forms are linear in their first argument. We
write $V'$ for the anti-dual of $V$, and the anti-duality pairing
$\langle\cdot,\cdot\rangle_{V',V}$ is linear in its first argument and
conjugate-linear in its second. Let
$V\hookrightarrow H\hookrightarrow V'$ be a Gelfand triple and let
$\mathcal A(t):V\to V'$ be the operator associated with a measurable family
of uniformly bounded, quasi-coercive forms with common domain $V$.
Then, for $f\in\Leb^2(0,T;V')$ and $u_0\in H$, Lions' theorem gives a unique
variational solution
\[
  u\in\Leb^2(0,T;V)\cap\Sob^1(0,T;V')
\]
of $\partial_tu+\mathcal A(t)u=f$ in $V'$ with $u(0)=u_0$; see
\cite{Lions1961}. For zero initial
data and $f\in\Leb^2(0,T;H)$, let $u$ be the corresponding variational
solution. For $t\in[0,T]$, define the realisation in $H$ by
\[
  D(A_H(t))
  :=\bigl\{v\in V:\mathcal A(t)v\in H\bigr\},
  \qquad
  A_H(t)v:=\mathcal A(t)v.
\]
We say that $u$ has maximal $\Leb^2$-regularity in $H$ if
\[
  \partial_tu\in\Leb^2(0,T;H),
  \qquad
  u(t)\in D(A_H(t))\quad\text{for a.e. }t,
  \qquad
  A_H(\cdot)u(\cdot)\in\Leb^2(0,T;H).
\]
Since
\[
  \mathcal A(t)u(t)=f(t)-\partial_tu(t)\quad\text{in }V'
\]
for almost every $t$, the first condition is equivalent to requiring both of the latter two conditions.

\subsection*{The \texorpdfstring{$\frac{1}{2}$}{1/2}-endpoint question}

It is well known that the answer to Lions' problem depends on the regularity of the form.
A classical result of Lions gives maximal regularity for symmetric $\C^1$
families and zero initial data; see \cite{Lions1961}. A central question is
whether the time dependence can be weakened to continuity, or even to
measurability. In the autonomous case the corresponding
result follows from de Simon's theorem \cite{deSimon1964}. In the general
common-domain form setting, with temporal regularity measured in
$\mathcal L(V,V')$ one half of a time derivative is the critical scale.
Ouhabaz and Spina proved maximal regularity
under $\C^{0,\alpha}$ dependence with $\alpha>1/2$
\cite{OuhabazSpina2010}. Further positive results cover symmetric forms of
bounded variation \cite{Dier2015}, suitable Dini moduli
\cite{HaakOuhabaz2015}, and fractional Sobolev regularity of order strictly
larger than $1/2$ \cite{DierZacher2017}; see also \cite{Trostorff2021}. We refer to
\cite{ArendtDierFackler2017} for a broader account of Lions' problem.

At the endpoint $1/2$, known positive results use additional structure beyond the estimate
\[
  \norm{\fra(t,\cdot)-\fra(s,\cdot)}_{\Leb^\infty(0,\pi)}
  \lesssim\abs{t-s}^{1/2}.
\]
For divergence-form operators, Auscher and Egert impose a scale-invariant
square condition pointwise in space, with a bound uniform in the spatial
variable \cite[Theorem~2, condition~(8)]{AuscherEgert2016}.
Achache and Ouhabaz obtained an abstract $\Leb^2$ result when the operator
path is piecewise $\Sob^{1/2}(0,T;\mathcal L(V,V'))$, together with the
uniform Kato square-root property and their local smallness
condition~(2.1) \cite[Theorem~2.2]{AchacheOuhabaz2019}.

Auscher and Egert conjectured that
$\C^{0,1/2}([0,T];\Leb^\infty(0,\pi))$ regularity is insufficient for
divergence-form operators \cite[Conjecture~14]{AuscherEgert2016}. Dier
constructed an abstract negative example for nonsymmetric,
discontinuous forms that exploits failure of the Kato square-root property
\cite[Section~5.2]{DierThesis2014}. Fackler subsequently
proved the failure for abstract symmetric forms with $\C^{0,1/2}$ time
dependence \cite[Theorem~5.1]{Fackler2017}. Both examples are abstract and
do not arise from differential operators. In one space dimension, Krylov
obtained earlier negative results for Sobolev solvability of scalar parabolic
equations with bounded measurable coefficients \cite{Krylov2016}. These
concern different $\Leb^p$-Sobolev formulations and do not address Lions'
$\Leb^2(0,T;\Leb^2(0,\pi))$ endpoint problem under
$\C^{0,1/2}([0,T];\Leb^\infty(0,\pi))$ regularity.

We prove that the obstruction occurs for a real scalar
coefficient on a bounded interval. It is neither the consequence of a loss of ellipticity nor that of
a large perturbation. We thus prove the conjecture of Auscher and
Egert \cite[Conjecture~14]{AuscherEgert2016}. The careful construction allows zero extension to the real line.
Tensorisation and localisation by parabolic rescaling extend the one-dimensional counterexample to real symmetric isotropic counterexamples
on $\R^d$ and on every bounded domain $\Omega\subset\R^d$, for every
$d\ge1$. In particular, the obstruction is neither a
boundary phenomenon nor related to the dimension or geometry of the underlying
domain.

The divergence-form counterexamples of Bechtel, Mooney, and Veraar use
non-Hermitian coefficients in dimension $d\ge2$, and their coefficient paths
are not continuous in the spatial $\Leb^\infty$ topology
\cite[Theorem~2.7 and Remark~2.8]{BechtelMooneyVeraar2024}. They explicitly
left open the symmetric or Hermitian case, the one-dimensional case, and the
case of continuous or H\"older-continuous coefficient paths. Theorem~\ref{thm:main}
resolves these issues simultaneously for a real scalar coefficient.

Before we state our main result, we introduce some helpful notation
\[
  H=\Leb^2((0,\pi);\CC),\qquad
  V=\Sob_0^1((0,\pi);\CC),\qquad V'=\Sob^{-1}((0,\pi);\CC).
\]
For $0<\alpha<1$ and a Banach space $X$, we use
$\C^{0,\alpha}([0,T];X)$ for the Banach-valued H\"older space, with norm
\[
  \norm{g}_{\C^{0,\alpha}([0,T];X)}
  :=
  \sup_{t\in[0,T]}\norm{g(t)}_X
  +
  \sup_{\substack{s,t\in[0,T]\\s\neq t}}
  \frac{\norm{g(t)-g(s)}_X}{\abs{t-s}^\alpha}.
\]
We write $[g]_{\C^{0,\alpha}([0,T];X)}$ for the second supremum.
Here, $\C^1$ denotes the space of continuously differentiable functions.
We use $\norm{v}_V=\norm{\partial_xv}_{\Leb^2(0,\pi)}$, which is an equivalent norm on $V$.
For a bounded real coefficient $\fra(t,x)$, we define
$\mathcal A(t):V\to V'$ by
\[
  \langle\mathcal A(t)v,w\rangle_{V',V}
  =\int_0^\pi \fra(t,x)\partial_xv(x)
  \overline{\partial_xw(x)}\dd x,
  \qquad v,w\in V.
\]
Distributionally,
\[
  \mathcal A(t)v=-\partial_x\bigl(\fra(t,\cdot)\partial_xv\bigr).
\]
We have
\[
  \norm{\mathcal A(t)-\mathcal A(s)}_{\mathcal L(V,V')}
  \le\norm{\fra(t,\cdot)-\fra(s,\cdot)}_{\Leb^\infty(0,\pi)}.
\]

\begin{theorem}[Counterexample on an interval]\label{thm:main}
  Let $\delta>0$. Then there exists a bounded, uniformly elliptic,
  real-valued coefficient $\fra$ satisfying
  \[
    \fra\in\C^{0,1/2}\bigl([0,1];\Leb^\infty(0,\pi)\bigr),
    \qquad
    \norm{\fra-1}_{\C^{0,1/2}
      ([0,1];\Leb^\infty(0,\pi))}<\delta,
  \]
  but
  \[
    \fra\notin\C^{0,\alpha}([0,1];\Leb^\infty(0,\pi))
    \qquad\text{for every }\frac12<\alpha<1.
  \]
  There also exist a real-valued forcing term
  $f\in\C([0,1];\Leb^2(0,\pi)) \cap \Leb^2(0,1;\Leb^2(0,\pi))$ and a real-valued function
  \[
    u\in\C([0,1];V)
    \cap\C^1([0,1];V')
    \cap W^{1,1}(0,1;H)
  \]
  such that $u(0)=0$ and $u$ is the unique Lions variational solution of
  \[
    \partial_tu(t)+\mathcal A(t)u(t)=f(t)
  \]
  in $V'$ for every $t\in[0,1]$. Moreover,
  \[
    u(t)\in D(A_H(t))\qquad\text{for every }t\in[0,1],
  \]
  and
  \[
    A_H(\cdot)u(\cdot)\in\Leb^1(0,1;H),
    \qquad
    A_H(\cdot)u(\cdot)\notin\Leb^2(0,1;H),
    \qquad
    \partial_tu\notin\Leb^2(0,1;H).
  \]
  In particular, maximal $\Leb^2$-regularity cannot hold for this diffusion
  coefficient.
\end{theorem}

The spatial boundary-vanishing conditions built into the interval construction
also allow the counterexample to be localised inside an arbitrary bounded domain. After zero
extension, tensorisation, and parabolic rescaling, we obtain the following
counterexample.

\begin{corollary}[Counterexamples on arbitrary bounded domains]
  \label{cor:bounded-domains}
  Let $d\ge1$, let $\Omega\subset\R^d$ be a bounded domain, and let
  $\delta>0$. Set
  \[
    H_\Omega=\Leb^2(\Omega;\CC),
    \qquad
    V_\Omega=\Sob_0^1(\Omega;\CC),
    \qquad
    V_\Omega'=\Sob^{-1}(\Omega;\CC).
  \]
  Then there exists a bounded real-valued scalar coefficient $\frb$ such that
  the isotropic coefficient matrix $\frB(t,x)=\frb(t,x)\Id_d$
  is real symmetric and uniformly elliptic, and satisfies
  \[
    \frB\in\C^{0,1/2}\bigl([0,1];
    \Leb^\infty(\Omega;\R^{d\times d})\bigr),
    \qquad
    \norm{\frB-\Id_d}_{\C^{0,1/2}
      ([0,1];\Leb^\infty(\Omega;\R^{d\times d}))}<\delta.
  \]
  Moreover,
  \[
    \frB\notin\C^{0,\alpha}\bigl([0,1];
    \Leb^\infty(\Omega;\R^{d\times d})\bigr)
    \qquad\text{for every }\frac12<\alpha<1.
  \]
  There also exist a real-valued forcing term
  \[
    F\in\C([0,1];H_\Omega)\cap\Leb^2(0,1;H_\Omega)
  \]
  and a real-valued function
  \[
    U\in\C([0,1];V_\Omega)
    \cap\C^1([0,1];V_\Omega')
    \cap W^{1,1}(0,1;H_\Omega)
  \]
  which is, in particular, an element of
  \[
    \Leb^2(0,1;V_\Omega)\cap\Sob^1(0,1;V_\Omega'),
  \]
  such that $U(0)=0$ and $U$ is the unique Lions variational solution of
  \[
    U'(t)-\operatorname{div}\bigl(\frB(t,\cdot)\nabla U(t)\bigr)=F(t)
  \]
  in $V_\Omega'$ for every $t\in[0,1]$, but
  \[
    U'\notin\Leb^2(0,1;H_\Omega).
  \]
  Consequently,
  \[
    -\operatorname{div}(\frB\nabla U)
    \notin\Leb^2(0,1;H_\Omega),
  \]
  and maximal $\Leb^2$-regularity cannot hold for this diffusion
  coefficient on $\Omega$.
\end{corollary}

\subsection*{Mechanism}

Let us now describe the idea of the construction. We first choose the coefficient, then find a candidate solution, and then define the source term as
\[
  f:=\partial_tu+\mathcal A(\cdot)u.
\]
The oscillatory components of the solution are chosen such that the bad
terms in this expression cancel, while their time derivatives are large
in $H$.

We write
\[
  A_{0,H}=-\partial_x^2,
  \qquad
  L_\kappa v=-\partial_x(\kappa\partial_xv).
\]
Let $K>2$ be an integer and set
\[
  \phi(x)=\sin^2x,
  \qquad
  p_K(x)=\sin(Kx)\sin(2x),
  \qquad
  c_K(x)=\cos(Kx)\cos(2x).
\]
The identity
\begin{equation}\label{eq:intro-profile-identity}
  L_{K^{-1}\cos(Kx)}\phi
  =p_K-2K^{-1}c_K
\end{equation}
shows that a small perturbation of the coefficient of size $K^{-1}$ produces an oscillatory
mode whose $H$-norm is independent of $K$. Indeed,
\[
  \norm{p_K}_H=\norm{c_K}_H=\frac{\sqrt\pi}{2},
  \qquad
  \norm{p_K}_{V'}\lesssim K^{-1}.
\]
Thus the divergence-form differential operator turns a coefficient perturbation of size \(K^{-1}\) into an order-one term in \(H\), while the same term is of order \(K^{-1}\) in \(V'\).

We modulate the small coefficient perturbation at temporal frequency $K^2$. Writing
\[
  b_K(t,x)=K^{-1}\sin(K^2t)\cos(Kx),
\]
we have, for $h=\abs{t-s}$,
\[
  \norm{b_K(t,\cdot)-b_K(s,\cdot)}_{\Leb^\infty(0,\pi)}
  \le\min\{2K^{-1},Kh\}
  \lesssim h^{1/2}.
\]
Hence the coefficient changes by size $K^{-1}$ on the parabolic time scale
$K^{-2}$. This relation is the origin of the H\"older exponent $\frac12$.

For the sake of simplicity, we omit the time cutoffs, which will be explained below. Put $\theta=K^2t$ and choose
\[
  y_K(t)
  =-\frac{\eps}{2K^2}\bigl(\sin\theta-\cos\theta\bigr).
\]
Using
\[
  A_{0,H}p_K=(K^2+4)p_K-4Kc_K,
\]
we obtain
\[
  \partial_ty_K+K^2y_K=-\eps\sin\theta,
  \qquad
  \abs{y_K}\lesssim\eps K^{-2},
  \qquad
  \abs{\partial_ty_K}\lesssim\eps,
\]
and the exact cancellation
\[
  (\partial_t+A_{0,H})(y_Kp_K)
  +\eps L_{b_K(t,\cdot)}\phi
  =
  4y_Kp_K-2\eps K^{-1}\cos\theta\,c_K.
\]
The two bad $p_K$-terms of size $\eps$ cancel. The remaining
expression has $H$-norm bounded by $C\eps K^{-1}$. Thus the corrector itself
has size $O(\eps K^{-2})$, but its time derivative has size
$O(\eps)$.

We next place these one-frequency constructions on disjoint time intervals
$I_j$ of length $\ell_j$. Since the intervals are contained in a finite time
interval, $\sum_j\ell_j<\infty$.
Consequently, one bounded mode on each interval would contribute at most
\[
  C\eps^2\sum_j\ell_j<\infty
\]
to the squared $\Leb^2(0,1;H)$-norm of the time derivative. Increasing the
size of a single mode does not help to make this infinite, since the
$\C^{0,1/2}$-seminorm of the corresponding coefficient atom increases by
the same factor.

Instead, we activate a growing number of mutually orthogonal modes on
each block. For the construction we will later choose
\[
  \ell_j=\frac1{16j(j+1)},
  \qquad
  M_j=j,
\]
and define the frequencies
\[
  K_{j,m}=16^m\ell_j^{-1},
  \qquad
  1\le m\le M_j.
\]
The lacunarity of the frequencies implies that the diffusion coefficient is
uniformly $\frac12$-H\"older continuous in time with values in
$\Leb^\infty$, independently of the number $M_j$ of active modes; see
Proposition~\ref{prop:coefficient}.

Let $J_j$ be the middle third of $I_j$, where the cutoff equals one. On
this interval,
\[
  \partial_tu(t)
  =-\frac{\eps}{2}\sum_{m=1}^{M_j}
  \bigl(\cos\theta_{j,m}(t)+\sin\theta_{j,m}(t)\bigr)
  p_{K_{j,m}}.
\]
Using the orthogonality of the spatial modes and integrating in time, we
deduce
\[
  \int_{I_j}\norm{\partial_tu(t)}_H^2\dd t
  \gtrsim\eps^2M_j\ell_j
  =\frac{\eps^2}{16(j+1)}.
\]
Here the factor $M_j$ comes from spatial orthogonality, whereas the factor
$\ell_j$ comes from time integration over the block. Hence
\[
  \sum_jM_j\ell_j=\infty, \qquad \partial_tu\notin\Leb^2(0,1;H).
\]

It remains to control the forcing. The estimate
$\norm{p_K}_{V'}\lesssim K^{-1}$ and the lacunarity of $K_{j,m}$ make sure that the solution
is an element of the variational class. After the modewise cancellation, the remainders are controllable.
We prove that
\[
  f=\partial_tu+\mathcal A(\cdot)u\in\C([0,1];H).
\]

Finally, $\phi$, $p_K$, and their first derivatives vanish at both
endpoints. Hence zero extension introduces no issues at the boundary of the support.
Tensorisation and parabolic rescaling then yield the full-space and
bounded-domain counterexamples. Section~\ref{sec:endpoint-exclusions}
compares the resulting coefficient with several known sufficient endpoint
hypotheses from the literature.

\subsection*{Notation}
For nonnegative quantities, $X\lesssim Y$ means $X\le CY$ with a constant
$C$ independent of the parameters under consideration, and $X\simeq Y$
means both $X\lesssim Y$ and $Y\lesssim X$.

\subsection*{Acknowledgements}
Lukas Niebel is funded by SNSF Starting Grant
TMSGI2\textunderscore226018 and by the Deutsche Forschungsgemeinschaft (DFG,
German Research Foundation) under Germany's Excellence Strategy
EXC 2044/2--390685587, Mathematics M\"unster:
Dynamics--Geometry--Structure.

\subsection*{Declaration of AI Use}
During an exploratory analysis, a first counterexample was found by OpenAI's
GPT-5.5 Pro in two dimensions and on the full space. It was then verified and studied
by the author, who simplified it and reduced it to the one-dimensional
interval counterexample presented here. OpenAI's GPT-5.6 Sol was
used for drafting and revision of parts of the exposition.
All mathematical claims, calculations, and AI-generated suggestions were
critically reviewed and verified by the author, who takes full responsibility
for the mathematical content and the final manuscript.

\section{Spatial oscillatory modes}

Let us first introduce the spatial oscillatory modes and record some identities
used in the construction. Let
\[
  A_{0,H}=-\partial_x^2,
  \qquad
  D(A_{0,H})=\Sob^2(0,\pi)\cap\Sob_0^1(0,\pi),
\]
be the Dirichlet Laplacian on $H$. For
$\kappa\in\Leb^\infty(0,\pi)$, we write
\[
  L_\kappa v=-\partial_x(\kappa\partial_xv),
  \qquad L_\kappa:V\longrightarrow V'.
\]
We choose the profile
\begin{equation}
  \phi(x)=\sin^2x.
\end{equation}
For every integer $K>2$, we set
\begin{equation}
  p_K(x)=\sin(Kx)\partial_x\phi(x)=\sin(Kx)\sin(2x),
  \qquad
  c_K(x)=\cos(Kx)\cos(2x).
\end{equation}
The profile $\phi$ and the modes $p_K$, together with their first derivatives,
vanish at both endpoints:
\begin{equation}\label{eq:endpoint-conditions}
  \begin{aligned}
    \phi(0) & =\partial_x\phi(0)=\phi(\pi)=\partial_x\phi(\pi)=0, \\
    \qquad
    p_K(0)  & =\partial_xp_K(0)=p_K(\pi)=\partial_xp_K(\pi)=0.
  \end{aligned}
\end{equation}
The vanishing endpoint values of $\phi$ and $p_K$ will imply that the solution satisfies the Dirichlet
condition in the interval construction. The vanishing of their first
derivatives is not needed in this case. It will be used later for the zero extension to
the real line in Section~\ref{sec:transfer}.

\begin{lemma}[Oscillatory-mode identities]\label{lem:oscillatory-modes}
  There is an absolute constant $C>0$ such that, for every integer $K>2$,
  \begin{equation}\label{eq:oscillatory-mode-norms}
    \norm{p_K}_H^2=\norm{c_K}_H^2=\frac\pi4,
    \qquad
    \norm{\partial_xp_K}_H^2=\frac\pi4(K^2+4),
    \qquad
    \norm{p_K}_{V'}\le CK^{-1}.
  \end{equation}
  Moreover,
  \begin{equation}\label{eq:oscillatory-mode-laplacian}
    A_{0,H}p_K=(K^2+4)p_K-4Kc_K.
  \end{equation}
  If $L\ge16K$, then the two-dimensional spaces
  $\operatorname{span}\{p_K,c_K\}$ and
  $\operatorname{span}\{p_L,c_L\}$ are orthogonal in $H$.
\end{lemma}

\begin{proof}
  We write
  \[
    p_K=\frac12\bigl(\cos((K-2)x)-\cos((K+2)x)\bigr),
    \quad
    c_K=\frac12\bigl(\cos((K-2)x)+\cos((K+2)x)\bigr).
  \]
  Orthogonality of the cosine functions on $(0,\pi)$ gives the $H$-norms.
  Differentiating the formula for $p_K$ we deduce the formula for the derivative
  norm. To estimate the negative norm, we define
  \[
    P_K(x)=\frac12\left(
    \frac{\sin((K-2)x)}{K-2}
    -\frac{\sin((K+2)x)}{K+2}
    \right).
  \]
  Then $\partial_xP_K=p_K$, and therefore, for $w\in V$,
  \[
    \abs{\int_0^\pi p_K\overline w\dd x}
    =\abs{-\int_0^\pi P_K\overline{\partial_xw}\dd x}
    \le\norm{P_K}_H\norm{w}_V
    \le CK^{-1}\norm{w}_V.
  \]
  Next, differentiation of $\sin(Kx)\sin(2x)$ gives
  \eqref{eq:oscillatory-mode-laplacian}. Finally, if $L\ge16K$, the sets
  $\{K-2,K+2\}$ and $\{L-2,L+2\}$ are disjoint, so orthogonality proves
  the last assertion.
\end{proof}

We recall the following identity for the action of the coefficient
on the low-frequency profile:
\begin{equation}\label{eq:coefficient-profile-identity}
  L_{K^{-1}\cos(Kx)}\phi
  =p_K-2K^{-1}c_K.
\end{equation}
The first term is the leading oscillatory mode, whereas the second term has the small
factor $K^{-1}$.

\section{Time blocks, frequencies, and the coefficient}

We now define the time blocks and the lacunary frequencies, and then construct
the coefficient. For $j\ge1$, we set
\begin{equation}
  \tau_j=\frac5{16}-\frac1{16j},
  \qquad
  \ell_j=\tau_{j+1}-\tau_j=\frac1{16j(j+1)},
  \qquad
  M_j=j.
\end{equation}
We write
\[
  I_j=(\tau_j,\tau_{j+1}).
\]
The intervals $I_j$ are adjacent and accumulate from the left at $5/16$.
We choose $\eta\in\C_c^\infty(0,1)$ such that
\[
  0\le\eta\le1,
  \qquad
  \supp\eta\subset[1/10,9/10],
  \qquad
  \eta=1\quad\text{on }[1/3,2/3],
\]
and define
\[
  \eta_j(t)=\eta\left(\frac{t-\tau_j}{\ell_j}\right),
  \qquad
  J_j=\left[\tau_j+\frac{\ell_j}{3},
    \tau_j+\frac{2\ell_j}{3}\right].
\]
Then $\eta_j=1$ on $J_j$ and
$\norm{\partial_t\eta_j}_\infty\le C\ell_j^{-1}$.

For $1\le m\le M_j$, we choose the integer frequencies
\begin{equation}
  K_{j,m}=16^m\ell_j^{-1}=16^{m+1}j(j+1),
  \qquad
  \theta_{j,m}(t)=K_{j,m}^2(t-\tau_j).
\end{equation}
Within each block, consecutive frequencies differ by a factor of $16$. The following properties will be repeatedly used:
\begin{equation}\label{eq:reciprocal-sums}
  \begin{aligned}
    S_j:=\sum_{m=1}^{M_j}K_{j,m}^{-1}
              & \le \frac{\ell_j}{15},    \\
    \sum_{m=1}^{M_j}K_{j,m}^{-2}
              & \le \frac{\ell_j^2}{255}, \\
    M_j\ell_j & =\frac1{16(j+1)},         \\
    M_jS_j    & \le\frac1{240(j+1)}.
  \end{aligned}
\end{equation}
We also have
\begin{equation}\label{eq:oscillations}
  K_{j,1}^2\ell_j=4096j(j+1)\longrightarrow\infty.
\end{equation}

\begin{remark}[Choice of the frequencies]
  The factor $16$ in the definition of $K_{j,m}$ is only used for
  summability and for the separation of the oscillatory modes. After increasing the first frequency and changing the numerical
  constants, one may also use a dyadic family.
\end{remark}

We define
\begin{equation}\label{eq:beta-def}
  \beta(t,x)=\sum_{j=1}^\infty\sum_{m=1}^{M_j}
  \eta_j(t)K_{j,m}^{-1}\sin\theta_{j,m}(t)\cos(K_{j,m}x).
\end{equation}
Here $(t,x)\in\R^2$. At a fixed time, at most one block in this sum is active and that
block contains finitely many modes, so \eqref{eq:beta-def} is a pointwise finite sum.
For later use, we write $\beta_j$ for its $j$-th block, extended by zero
outside $I_j$ in time, so that $\beta=\sum_j\beta_j$.
Given $0<\eps\le1$, we set on $\R^2$
\begin{equation}\label{eq:fra-def}
  \fra(t,x)=1+\eps\beta(t,x).
\end{equation}
For the interval problem we use the restriction to
$[0,1]\times(0,\pi)$.

\begin{proposition}[Ellipticity and $\frac{1}{2}$-H\"older time regularity]
  \label{prop:coefficient}
  The coefficient $\fra$ is bounded, real-valued, and satisfies
  \[
    \norm{\fra-1}_{\Leb^\infty(\R^2)}\le C\eps.
  \]
  Moreover,
  \[
    \norm{\fra(t,\cdot)-\fra(s,\cdot)}_{\Leb^\infty(\R)}
    \le C\eps\abs{t-s}^{1/2},
    \qquad s,t\in\R.
  \]
  Here $C>0$ depends only on the fixed cutoff $\eta$.
  Choosing $\eps$ sufficiently small guarantees that $\fra$ is
  uniformly elliptic.
\end{proposition}

\begin{proof}
  We first prove the uniform estimate. At each fixed time only one block
  contributes, and therefore
  \[
    \norm{\beta(t,\cdot)}_\infty
    \le\sup_jS_j\le\frac{\ell_1}{15}.
  \]
  Hence $\norm{\fra-1}_{\Leb^\infty(\R^2)}\le C\eps$, and $\fra$ is uniformly
  elliptic when $\eps$ is sufficiently small.
  Put $h=\abs{t-s}$.
  If $h\le\ell_j$, then, using
  $\norm{\partial_t\eta_j}_\infty\le C\ell_j^{-1}$ and
  \[
    K_{j,m}^{-1}
    \abs{\sin\theta_{j,m}(t)-\sin\theta_{j,m}(s)}
    \le
    \min\left\{\frac2{K_{j,m}},K_{j,m}h\right\},
  \]
  we obtain
  \[
    \norm{\beta_j(t)-\beta_j(s)}_\infty
    \le
    \frac{Ch}{\ell_j}S_j
    +\sum_{m=1}^{M_j}
    \min\left\{\frac2{K_{j,m}},K_{j,m}h\right\}.
  \]
  The first term is bounded by $Ch$ by
  \eqref{eq:reciprocal-sums}. To estimate the second one, the case $h=0$
  being clear, set $K_h=h^{-1/2}$ and split the sum according to whether
  $K_{j,m}\le K_h$ or $K_{j,m}>K_h$. Since
  $K_{j,m+1}=16K_{j,m}$, the corresponding geometric sums satisfy
  \[
    h\sum_{K_{j,m}\le K_h}K_{j,m}
    \le \frac{16}{15}hK_h
    =\frac{16}{15}h^{1/2},
    \qquad
    2\sum_{K_{j,m}>K_h}K_{j,m}^{-1}
    \le\frac{32}{15}K_h^{-1}
    =\frac{32}{15}h^{1/2},
  \]
  with the usual convention that an empty sum is zero. Consequently,
  \[
    \sum_{m=1}^{M_j}
    \min\left\{\frac2{K_{j,m}},K_{j,m}h\right\}
    \le Ch^{1/2}.
  \]
  Combining these estimates and using $h\le\ell_j\le1$, we conclude that
  \[
    \norm{\beta_j(t)-\beta_j(s)}_\infty
    \le Ch+Ch^{1/2}
    \le Ch^{1/2}.
  \]
  If $h>\ell_j$, then
  \[
    \norm{\beta_j(t)-\beta_j(s)}_\infty
    \le2S_j\le C\ell_j\le Ch^{1/2}.
  \]
  Thus every zero-extended block has a uniform $1/2$-H\"older bound.

  Suppose now that $t$ and $s$ belong to different blocks $j$ and $k$,
  respectively. Then
  \[
    \beta(t)-\beta(s)
    =\bigl(\beta_j(t)-\beta_j(s)\bigr)
    +\bigl(\beta_k(t)-\beta_k(s)\bigr).
  \]
  The two uniform one-block estimates give the desired global bound. The
  same argument, with one or both summands omitted, covers times at which
  $\beta$ vanishes. Multiplying by $\eps$ completes the proof.
\end{proof}

\begin{proposition}[Sharpness within the H\"older scale]
  \label{prop:holder-sharp}
  For every $1/2<\alpha<1$, the coefficient satisfies
  \[
    \fra\notin\C^{0,\alpha}([0,1];\Leb^\infty(0,\pi)).
  \]
\end{proposition}

\begin{proof}
  Let $K\ge1$ be an integer and consider the Fourier coefficient
  \[
    \Lambda_K(\zeta)=\frac2\pi\int_0^\pi \zeta(x)\cos(Kx)\dd x,
    \qquad
    \abs{\Lambda_K(\zeta)}\le2\norm{\zeta}_{\Leb^\infty(0,\pi)}.
  \]
  Put $K=K_{j,1}$ and $h_j=\pi/(2K^2)$. By
  \eqref{eq:oscillations},
  $K^2\abs{J_j}=K^2\ell_j/3\to\infty$. Hence, for all sufficiently large
  $j$, we can choose an integer $n_j$ such that
  \[
    t_j=\tau_j+\frac{n_j\pi}{K^2},
    \qquad t_j,\ t_j+h_j\in J_j.
  \]
  Thus
  $\eta_j(t_j)=\eta_j(t_j+h_j)=1$, and spatial orthogonality removes every
  mode except the first. Hence
  \[
    \abs{\Lambda_K\bigl(\fra(t_j+h_j)-\fra(t_j)\bigr)}
    =\frac{\eps}{K},
    \qquad
    \norm{\fra(t_j+h_j)-\fra(t_j)}_\infty\ge\frac{\eps}{2K}.
  \]
  The corresponding $\alpha$-H\"older quotient satisfies
  \[
    \frac{\norm{\fra(t_j+h_j)-\fra(t_j)}_\infty}{h_j^\alpha}
    \ge 2^{\alpha-1}\pi^{-\alpha}\eps K^{2\alpha-1},
  \]
  which tends to infinity. This rules out every H\"older exponent
  $1/2<\alpha<1$.
\end{proof}

\begin{proposition}[Common operator domain without uniform graph control]
  \label{prop:frozen-domains}
  Let
  \[
    D_0=D(A_{0,H})=\Sob^2(0,\pi)\cap\Sob_0^1(0,\pi),
  \]
  endowed with the graph norm of $A_{0,H}$. Then
  \[
    D(A_H(t))=D_0\qquad(t\in[0,1]),
  \]
  but
  \[
    \sup_{t\in[0,1]}
    \norm{A_H(t)}_{\mathcal L(D_0,H)}=\infty.
  \]
\end{proposition}

\begin{proof}
  Fix $t$ and write $\fra_t=\fra(t,\cdot)$. Since only finitely many
  spatial modes are active at each time and $\fra$ is uniformly elliptic,
  \[
    \fra_t,\fra_t^{-1}\in W^{1,\infty}(0,\pi).
  \]
  The product rule gives $D_0\subseteq D(A_H(t))$ and, for $v\in D_0$,
  \[
    \norm{A_H(t)v}_H
    \le\norm{\fra_t}_\infty\norm{\partial_x^2v}_H
    +\norm{\partial_x\fra_t}_\infty\norm{\partial_xv}_H
    \le C_t\norm{v}_{D_0}.
  \]
  Thus $A_H(t)\in\mathcal L(D_0,H)$. Conversely, if
  $v\in D(A_H(t))$, then $q=\fra_t\partial_xv$ satisfies
  \[
    q\in H,
    \qquad
    \partial_xq=-A_H(t)v\in H.
  \]
  Hence $q\in\Sob^1(0,\pi)$, so
  $\partial_xv=\fra_t^{-1}q\in\Sob^1(0,\pi)$, proving
  $v\in D_0$.

  To see that the control is not uniform, take $\phi(x)=\sin^2x$. For
  $t\in J_j$, the construction gives
  \[
    L_{\beta(t,\cdot)}\phi
    =
    \sum_{m=1}^{M_j}\sin\theta_{j,m}(t)
    \bigl(p_{K_{j,m}}-2K_{j,m}^{-1}c_{K_{j,m}}\bigr).
  \]
  Lemma~\ref{lem:oscillatory-modes} gives orthogonality between distinct
  summands, and
  \[
    \norm{p_K-2K^{-1}c_K}_H^2
    =\frac\pi4\bigl(1+4K^{-2}\bigr)\ge\frac\pi4.
  \]
  Moreover, $\abs{J_j}=\ell_j/3$ and
  $K_{j,m}^2\ell_j\ge6$, so
  \[
    \frac1{\abs{J_j}}\int_{J_j}
    \sin^2\theta_{j,m}(t)\dd t
    \ge\frac12-\frac1{2K_{j,m}^2\abs{J_j}}
    \ge\frac14.
  \]
  Averaging over $J_j$ therefore yields some $t_j\in J_j$ such that
  \[
    \norm{L_{\beta(t_j,\cdot)}\phi}_H^2
    \ge \frac{\pi}{16}M_j.
  \]
  Since $M_j=j$ and
  \[
    A_H(t_j)\phi
    =A_{0,H}\phi+\eps L_{\beta(t_j,\cdot)}\phi,
  \]
  we obtain
  \[
    \norm{A_H(t_j)\phi}_H
    \ge \frac{\eps\sqrt\pi}{4}\sqrt j
    -\norm{A_{0,H}\phi}_H
    \longrightarrow\infty.
  \]
  Since $\norm{\phi}_{D_0}$ is fixed, it follows that
  $\norm{A_H(t_j)}_{\mathcal L(D_0,H)}\to\infty$.
\end{proof}

\section{The variational solution and its time derivative}

Next, we construct the variational solution and isolate the divergent
part of its time derivative. We choose $\rho\in\C_c^\infty(0,1)$ such that
\begin{equation}
  \rho=1\quad\text{on }[1/4,1/2].
\end{equation}
In particular, $\rho(0)=0$, and $\rho=1$ on a neighbourhood of the closure of
$\bigcup_j\supp\eta_j$.

For each mode, we define the real-valued scalar function
\begin{equation}\label{eq:y-def}
  y_{j,m}(t)
  =-\frac{\eps}{2K_{j,m}^2}\eta_j(t)
  \bigl(\sin\theta_{j,m}(t)-\cos\theta_{j,m}(t)\bigr).
\end{equation}
Differentiation gives
\begin{equation}\label{eq:y-ode}
  \partial_ty_{j,m}+K_{j,m}^2y_{j,m}
  =-\eps\eta_j\sin\theta_{j,m}+r_{j,m},
\end{equation}
where
\begin{equation}\label{eq:r-def}
  r_{j,m}
  =-\frac{\eps}{2K_{j,m}^2}\partial_t\eta_j
  (\sin\theta_{j,m}-\cos\theta_{j,m}),
  \qquad
  \abs{r_{j,m}}\le C\eps K_{j,m}^{-2}\ell_j^{-1}.
\end{equation}
We then set
\begin{equation}\label{eq:u-def}
  v(t,x)=\sum_{j=1}^\infty\sum_{m=1}^{M_j}
  y_{j,m}(t)p_{K_{j,m}}(x),
  \qquad
  u(t,x)=\rho(t)\phi(x)+v(t,x).
\end{equation}
Again, the sum is finite at each time.

\begin{lemma}[Regularity and failure of $\Leb^2$-integrability of the time derivative]
  \label{lem:solution-regularity}
  The function $u$ satisfies
  \[
    u\in\C([0,1];V)\cap\C^1([0,1];V')
    \cap W^{1,1}(0,1;H),
    \qquad u(0)=0,
  \]
  and
  \[
    \partial_tu\notin\Leb^2(0,1;H).
  \]
\end{lemma}

\begin{proof}
  Let us first prove the stated regularity. Write
  $v_j=\sum_{m=1}^{M_j}y_{j,m}p_{K_{j,m}}$, with each $v_j$
  extended by zero outside $I_j$. Equations \eqref{eq:y-def} and
  \eqref{eq:r-def}, together with $K_{j,m}^2\ell_j\ge1$, give the pointwise
  estimates
  \[
    \abs{y_{j,m}(t)}\le C\eps K_{j,m}^{-2},
    \qquad
    \abs{\partial_ty_{j,m}(t)}\le C\eps.
  \]
  Hence Lemma~\ref{lem:oscillatory-modes} and \eqref{eq:reciprocal-sums} imply
  \begin{equation}\label{eq:block-vvprime-pointwise}
    \norm{v_j}_{\C(\R;V)}
    +\norm{\partial_tv_j}_{\C(\R;V')}
    \le C\eps\sum_{m=1}^{M_j}K_{j,m}^{-1}
    \le C\eps\ell_j.
  \end{equation}
  Since $\sum_j\ell_j=1/16$, estimate
  \eqref{eq:block-vvprime-pointwise} shows that the series $\sum_jv_j$ and
  $\sum_j\partial_tv_j$ converge uniformly in $V$ and $V'$, respectively.
  Each zero-extended block is smooth in time. Applying
  \cite[Chap.~V, \S~2, Thm.~2.8]{AmannEscher2005}
  to the partial sums, viewed as $V'$-valued maps, gives
  \[
    \partial_tv=\sum_{j=1}^\infty \partial_tv_j
    \quad\text{in }\C(\R;V').
  \]
  Together with the uniform convergence of $\sum_jv_j$ in $V$, this yields
  \[
    v\in\C([0,1];V)\cap\C^1([0,1];V').
  \]
  The same is true for the fixed term $\rho\phi$. Both terms vanish at $t=0$.

  We next prove the additional time regularity in $H$. Orthogonality of the
  oscillatory modes, the bound $\abs{\partial_ty_{j,m}}\le C\eps$, and
  Lemma~\ref{lem:oscillatory-modes} give
  \[
    \norm{\partial_tv_j(t)}_H^2
    =\frac\pi4\sum_{m=1}^{M_j}\abs{\partial_ty_{j,m}(t)}^2
    \le C\eps^2M_j.
  \]
  Consequently,
  \[
    \norm{\partial_tv_j}_{\Leb^1(I_j;H)}
    \le C\eps\ell_j\sqrt{M_j}.
  \]
  Moreover,
  \[
    \sum_{j=1}^\infty\ell_j\sqrt{M_j}
    =\frac1{16}\sum_{j=1}^\infty
    \frac1{\sqrt j\,(j+1)}<\infty.
  \]
  Thus $\sum_j\partial_tv_j$ converges in $\Leb^1(0,1;H)$. Since its sum is already
  the distributional time derivative of $v$, it follows that
  $v\in W^{1,1}(0,1;H)$. The smooth background term has the same regularity,
  and hence $u\in W^{1,1}(0,1;H)$.

  It remains to prove that the time derivative is not square integrable. On $J_j$ we have
  $\partial_t\rho=0$, $\eta_j=1$, and $\partial_t\eta_j=0$, and thus
  \begin{equation}\label{eq:partial-t-u-on-J}
    \partial_tu(t)
    =-\frac\eps2\sum_{m=1}^{M_j}
    \bigl(\cos\theta_{j,m}(t)+\sin\theta_{j,m}(t)\bigr)
    p_{K_{j,m}}.
  \end{equation}
  In particular, \eqref{eq:partial-t-u-on-J} represents $\partial_tu(t)$ as an element of
  $H$ for every $t\in J_j$. The oscillatory modes are mutually orthogonal by
  Lemma~\ref{lem:oscillatory-modes}. Moreover,
  \[
    \int_{J_j}(\cos\theta_{j,m}+\sin\theta_{j,m})^2\dd t
    \ge\abs{J_j}-K_{j,m}^{-2}
    =\frac{\ell_j}{3}-K_{j,m}^{-2}
    \ge\frac{\ell_j}{6}.
  \]
  Here the last inequality follows from
  $K_{j,m}^2\ell_j\ge K_{j,1}^2\ell_j\ge6$.
  Since $\norm{p_K}_H^2=\pi/4$, we conclude
  \[
    \int_0^1\norm{\partial_tu(t)}_H^2\dd t
    \ge\frac{\pi\eps^2}{96}\sum_{j=1}^\infty M_j\ell_j
    =\frac{\pi\eps^2}{1536}\sum_{j=1}^\infty\frac1{j+1}
    =\infty.
  \]
\end{proof}

\section{Cancellation and the forcing}

We now exploit the oscillatory-mode cancellation to show that the forcing is
continuous with values in $H$. For the coefficient \eqref{eq:fra-def}, the
associated form operator has the form
\[
  \mathcal A(t)=L_{\fra(t,\cdot)}=L_1+\eps L_{\beta(t,\cdot)}.
\]
The choice of $y_{j,m}$ in \eqref{eq:y-def} makes the leading
$p_K$-terms cancel mode by mode. For a fixed pair $j,m$, write
$K=K_{j,m}$. Since $p_K$ is independent of time,
\eqref{eq:oscillatory-mode-laplacian} and \eqref{eq:y-ode} give
\begin{align*}
  (\partial_t+A_{0,H})(y_{j,m}p_K)
   & =(\partial_ty_{j,m})p_K
  +y_{j,m}\bigl((K^2+4)p_K-4Kc_K\bigr)           \\
   & =(\partial_ty_{j,m}+K^2y_{j,m}+4y_{j,m})p_K
  -4Ky_{j,m}c_K                                  \\
   & =\bigl(-\eps\eta_j\sin\theta_{j,m}
  +r_{j,m}+4y_{j,m}\bigr)p_K
  -4Ky_{j,m}c_K.
\end{align*}
The coefficient--profile term has the opposite leading $p_K$-part.
Indeed, $\rho\eta_j=\eta_j$, since $\rho=1$ on $\supp\eta_j$, and
\eqref{eq:coefficient-profile-identity} yields
\begin{align*}
   & \eps L_{\eta_jK^{-1}\sin\theta_{j,m}\cos(Kx)}(\rho\phi) \\
   & \qquad={}
  \eps\eta_j\sin\theta_{j,m}
  L_{K^{-1}\cos(Kx)}\phi                                     \\
   & \qquad={}
  \eps\eta_j\sin\theta_{j,m}p_K
  -2\eps\eta_jK^{-1}\sin\theta_{j,m}c_K.
\end{align*}
Adding these identities, the two leading $p_K$-terms cancel, and we obtain
\begin{align*}
   & (\partial_t+A_{0,H})(y_{j,m}p_K)
  +\eps L_{\eta_jK^{-1}\sin\theta_{j,m}\cos(Kx)}(\rho\phi) \\
   & \qquad={}
  (r_{j,m}+4y_{j,m})p_K
  +\bigl(-4Ky_{j,m}
  -2\eps\eta_jK^{-1}\sin\theta_{j,m}\bigr)c_K.
\end{align*}
Finally, \eqref{eq:y-def} gives
\begin{align*}
  -4Ky_{j,m}
  -2\eps\eta_jK^{-1}\sin\theta_{j,m}
   & =2\eps\eta_jK^{-1}
  \bigl(\sin\theta_{j,m}-\cos\theta_{j,m}\bigr)
  -2\eps\eta_jK^{-1}\sin\theta_{j,m}      \\
   & =-2\eps\eta_jK^{-1}\cos\theta_{j,m}.
\end{align*}
We collect the remaining terms in
\begin{equation}\label{eq:E-def}
  E_{j,m}
  =(r_{j,m}+4y_{j,m})p_K
  -2\eps\eta_jK^{-1}\cos\theta_{j,m}\,c_K.
\end{equation}
We have proven
\begin{equation}\label{eq:exact-mode-cancellation}
  (\partial_t+A_{0,H})(y_{j,m}p_K)
  +\eps L_{\eta_jK^{-1}\sin\theta_{j,m}\cos(Kx)}(\rho\phi)
  =E_{j,m}.
\end{equation}

We set, for $t\in[0,1]$,
\[
  R(t)=\sum_{j=1}^\infty\sum_{m=1}^{M_j}E_{j,m}(t),
  \qquad
  f_0(t)=(\partial_t\rho)(t)\phi+\rho(t)A_{0,H}\phi,
\]
and define, initially as a continuous $V'$-valued function,
\begin{equation}\label{eq:f-def}
  f(t):=\partial_tu(t)+\mathcal A(t)u(t),
  \qquad t\in[0,1].
\end{equation}
This is well-defined because Lemma~\ref{lem:solution-regularity} gives
$u\in\C([0,1];V)$ and $\partial_tu\in\C([0,1];V')$, while
Proposition~\ref{prop:coefficient} and the form estimate give
$\mathcal A\in\C([0,1];\mathcal L(V,V'))$. Hence
$t\mapsto\mathcal A(t)u(t)$ is continuous
with values in $V'$. Summing
\eqref{eq:exact-mode-cancellation} yields
\begin{equation}\label{eq:f-decomposition}
  f(t)=f_0(t)+R(t)+\eps L_{\beta(t,\cdot)}v(t),
  \qquad t\in[0,1].
\end{equation}

\begin{proposition}[Continuous $H$-valued forcing]\label{prop:forcing}
  The function $f$ defined by \eqref{eq:f-def} belongs to
  $\C([0,1];H)$ and satisfies
  \[
    \partial_tu(t)+\mathcal A(t)u(t)=f(t)
  \]
  in $V'$ for every $t\in[0,1]$.
\end{proposition}

\begin{proof}
  We first treat the remainder. Let $R_j=\sum_mE_{j,m}$. From
  \eqref{eq:E-def}, \eqref{eq:r-def}, and \eqref{eq:reciprocal-sums},
  \begin{align*}
    \norm{R_j(t)}_H
     & \le C\eps\left(
    \ell_j^{-1}\sum_mK_{j,m}^{-2}
    +\sum_mK_{j,m}^{-2}+S_j\right) \\
     & \le C\eps\ell_j.
  \end{align*}
  Each zero-extended $R_j$ is a continuous, compactly supported $H$-valued
  function on $I_j$. The supports are disjoint and $\ell_j\to0$, so the
  series defining $R$ converges uniformly in $H$. Thus
  \begin{equation}
    R\in\C([0,1];H).
  \end{equation}

  Next, we estimate $\eps L_{\beta(t,\cdot)}v(t)$. If $t\in I_j$, the sums contain
  only the $M_j$ modes of that block, and
  \[
    \norm{\beta(t)}_\infty\le S_j,
    \qquad
    \norm{\partial_x\beta(t)}_\infty\le M_j,
  \]
  while Lemma~\ref{lem:oscillatory-modes} and \eqref{eq:y-def} give
  \[
    \norm{\partial_xv(t)}_H\le C\eps S_j,
    \qquad
    \norm{\partial_x^2v(t)}_H\le C\eps M_j.
  \]
  In the second estimate we used
  $\norm{\partial_x^2p_K}_H\le CK^2$, which follows immediately from
  \eqref{eq:oscillatory-mode-laplacian} and \eqref{eq:oscillatory-mode-norms}.
  The ordinary product rule therefore yields
  \begin{align*}
    \norm{\eps L_{\beta(t,\cdot)}v(t)}_H
     & \le\eps\norm{\partial_x\beta(t)}_\infty
    \norm{\partial_xv(t)}_H
    +\eps\norm{\beta(t)}_\infty
    \norm{\partial_x^2v(t)}_H                  \\
     & \le C\eps^2M_jS_j
    \le\frac{C\eps^2}{j+1}.
  \end{align*}
  For each $j$, the zero extension of
  $t\mapsto\eps L_{\beta_j(t,\cdot)}v_j(t)$ is continuous and compactly supported in $I_j$.
  Since these supports are disjoint,
  \[
    \sup_{t\in[0,1]}
    \norm{\sum_{j>N}\eps L_{\beta_j(t,\cdot)}v_j(t)}_H
    \le\sup_{j>N}\frac{C\eps^2}{j+1}\longrightarrow0.
  \]
  Thus their sum belongs to $\C([0,1];H)$. Finally,
  $f_0\in\C([0,1];H)$, and hence
  \eqref{eq:f-decomposition} proves the proposition.
\end{proof}

\begin{proof}[Proof of Theorem~\ref{thm:main}]
  Proposition~\ref{prop:coefficient} gives
  \[
    \norm{\fra-1}_{\Leb^\infty((0,1)\times(0,\pi))}\le C\eps,
    \qquad
    [\fra]_{\C^{0,1/2}([0,1];\Leb^\infty(0,\pi))}\le C\eps.
  \]
  Hence
  \[
    \norm{\fra-1}_{\C^{0,1/2}([0,1];\Leb^\infty(0,\pi))}
    \le C\eps.
  \]
  Choosing $\eps>0$ sufficiently small gives both uniform ellipticity and
  the prescribed bound $\delta$ in the full $\C^{0,1/2}$-norm.
  Proposition~\ref{prop:holder-sharp} proves the asserted H\"older sharpness.

  Next, Lemma~\ref{lem:solution-regularity} and Proposition~\ref{prop:forcing} show that
  $u$ is a variational solution with zero initial value and continuous
  $H$-valued forcing. The definitions of $\beta$, $y_{j,m}$, and $f$ show
  that $\fra$, $u$, and $f$ are real-valued. Moreover, the forms associated
  with $\mathcal A(t)$ are measurable, uniformly bounded, and uniformly coercive,
  so uniqueness follows from Lions' theorem.

  Finally, for every $t\in[0,1]$, the sum defining $v(t)$ is finite, and
  $\phi,p_K\in D_0$. Hence Proposition~\ref{prop:frozen-domains} gives
  $u(t)\in D(A_H(t))$ for every $t\in[0,1]$. By
  Lemma~\ref{lem:solution-regularity},
  $\partial_tu\in\Leb^1(0,1;H)$, while
  Proposition~\ref{prop:forcing} gives
  $f\in\C([0,1];H)\subset\Leb^1(0,1;H)\cap\Leb^2(0,1;H)$.
  Thus, almost everywhere,
  \[
    A_H(t)u(t)=f(t)-\partial_tu(t),
  \]
  and consequently
  $A_H(\cdot)u(\cdot)\in\Leb^1(0,1;H)$. The same lemma gives
  $\partial_tu\notin\Leb^2(0,1;H)$. If $A_H(\cdot)u(\cdot)$ belonged to
  $\Leb^2(0,1;H)$, then
  $\partial_tu=f-A_H(\cdot)u(\cdot)$ would belong to
  $\Leb^2(0,1;H)$, contradicting
  Lemma~\ref{lem:solution-regularity}.
  This completes the proof.
\end{proof}

\begin{remark}[Growth of the maximal-regularity constant]
  For $N\in\mathbb N$, let
  \[
    \beta^{(N)}=\sum_{j=1}^N\beta_j,
    \qquad
    \fra_N=1+\eps\beta^{(N)},
    \qquad
    v^{(N)}=\sum_{j=1}^Nv_j,
    \qquad
    u_N=\rho\phi+v^{(N)}.
  \]
  We denote by
  \[
    \mathcal A_N(t)=L_{\fra_N(t,\cdot)}
  \]
  the corresponding form operator and set
  \[
    f_N(t)=\partial_tu_N(t)+\mathcal A_N(t)u_N(t).
  \]
  These functions are finite sums of smooth functions. Hence the problem
  associated with $\mathcal A_N$ has maximal $\Leb^2$-regularity. Moreover,
  the estimates above hold uniformly in $N$ and give
  \[
    \sup_{N\in\mathbb N}
    \norm{\fra_N-1}_{\C^{0,1/2}
      ([0,1];\Leb^\infty(0,\pi))}
    \le C\eps,
    \qquad
    \sup_{N\in\mathbb N}
    \norm{f_N}_{\Leb^2(0,1;H)}
    <\infty.
  \]
  On the other hand, the lower bound on the middle thirds of the blocks
  yields
  \[
    \norm{\partial_tu_N}_{\Leb^2(0,1;H)}^2
    \ge
    \frac{\pi\eps^2}{1536}
    \sum_{j=1}^N\frac1{j+1}
    \ge c\eps^2\log(N+1).
  \]

  Let $C_N$ be the best constant for which
  \[
    \norm{\partial_tw}_{\Leb^2(0,1;H)}
    +
    \norm{A_{N,H}(\cdot)w(\cdot)}_{\Leb^2(0,1;H)}
    \le
    C_N\norm{g}_{\Leb^2(0,1;H)}
  \]
  holds for every zero-initial-value solution of $\partial_tw+\mathcal A_N(t)w=g$.
  Applying this estimate to $(u_N,f_N)$ shows that, for every fixed
  $\eps>0$,
  \[
    C_N\ge c\eps\sqrt{\log(N+1)}.
  \]
\end{remark}

\section{Full-space and bounded-domain consequences}\label{sec:transfer}

In this section, we first extend the interval counterexample to the real line.
Afterwards, we tensorise it to construct counterexamples on $\R^d$.
Finally, we localise the full-space construction to give a counterexample on any bounded domain.
Here, the key ingredient is that $\phi$, $p_K$ and their first derivatives vanish at both endpoints.

For a function $z$ on $(0,\pi)$, let $E_0z$ denote its zero extension to
$\R$. For time-dependent functions, $E_0$ acts in the spatial variable.

\begin{lemma}[Zero extension of a vanishing flux]
  \label{lem:zero-extension-flux}
  Let
  \[
    q\in\Leb^1\bigl(0,1;\Sob^1(0,\pi)\bigr)
  \]
  and suppose that
  \[
    q(t,0)=q(t,\pi)=0
  \]
  for almost every $t\in(0,1)$. Then
  \[
    \partial_x(E_0q)=E_0(\partial_xq)
  \]
  in $\mathcal D'((0,1)\times\R)$. In particular, the spatial derivative of
  the flux extended by zero contains no endpoint Dirac masses.
\end{lemma}

\begin{proof}
  Let $\psi\in\C_c^\infty((0,1)\times\R)$. Integration by parts in the
  spatial variable gives
  \begin{align*}
    \left\langle\partial_x(E_0q),\psi\right\rangle
     & =-\int_0^1\int_0^\pi q(t,x)\partial_x\psi(t,x)\dd x\dd t \\
     & =\int_0^1\int_0^\pi \partial_xq(t,x)\psi(t,x)\dd x\dd t  \\
     & \quad+\int_0^1
    \bigl(q(t,0)\psi(t,0)-q(t,\pi)\psi(t,\pi)\bigr)\dd t.
  \end{align*}
  The last integral vanishes by the trace assumptions.
\end{proof}

\begin{proposition}[Counterexample on the real line]\label{prop:line}
  Set $\widetilde u:=E_0u$ and $\widetilde f:=E_0f$. With the coefficient
  \eqref{eq:fra-def}, the function $\widetilde u$ is the Lions variational solution
  on $\R$ with forcing $\widetilde f$. Moreover, $\fra$ is bounded, uniformly
  elliptic, and belongs to $\C^{0,1/2}([0,1];\Leb^\infty(\R))$, but
  \[
    \partial_t\widetilde u\notin\Leb^2(0,1;\Leb^2(\R)).
  \]
\end{proposition}

\begin{proof}
  The coefficient properties follow from
  Proposition~\ref{prop:coefficient}.

  By Lemma~\ref{lem:solution-regularity},
  \[
    u\in W^{1,1}(0,1;H).
  \]
  Set
  \[
    q(t,\cdot):=\fra(t,\cdot)\partial_xu(t,\cdot).
  \]
  Boundedness of $\fra$ and the fact that $u\in\C([0,1];V)$ give
  \[
    q\in\Leb^1(0,1;H).
  \]
  The equation on $(0,\pi)$ yields
  \[
    \partial_xq=\partial_tu-f\in\Leb^1(0,1;H).
  \]
  Consequently,
  \[
    q\in\Leb^1\bigl(0,1;\Sob^1(0,\pi)\bigr).
  \]

  From \eqref{eq:u-def} and \eqref{eq:endpoint-conditions},
  \[
    \partial_xu(t,0)=\partial_xu(t,\pi)=0,
  \]
  since $\partial_x\phi(0)=\partial_x\phi(\pi)=0$ and
  $\partial_xp_K(0)=\partial_xp_K(\pi)=0$. Hence
  \[
    q(t,0)=q(t,\pi)=0
  \]
  for almost every $t$. Lemma~\ref{lem:zero-extension-flux} therefore gives
  \[
    \partial_x(E_0q)=E_0(\partial_xq)
    \quad\text{in }\mathcal D'((0,1)\times\R).
  \]

  The map $E_0:H\to\Leb^2(\R)$ is bounded. It maps $W^{1,1}(0,1;H)$ into
  $W^{1,1}(0,1;\Leb^2(\R))$ and commutes with the weak time derivative.
  Thus
  \[
    \partial_t\widetilde u=E_0(\partial_tu).
  \]
  Separately, the zero trace of $u$ gives
  \[
    \partial_x\widetilde u=E_0(\partial_xu).
  \]
  Since $\fra$ is defined on the whole real line,
  \[
    \fra\,\partial_x\widetilde u=E_0q.
  \]
  It follows that
  \[
    \partial_t\widetilde u
    -\partial_x\bigl(\fra\,\partial_x\widetilde u\bigr)
    =E_0(\partial_tu)-E_0(\partial_xq)
    =E_0f
    =\widetilde f
  \]
  in $\mathcal D'((0,1)\times\R)$. For
  $z,w\in\Sob^1(\R)$, the maps
  \[
    (z,w)\longmapsto
    \int_{\R}\fra(t,x)\partial_xz(x)\overline{\partial_xw(x)}\dd x
  \]
  define forms that are uniformly bounded and uniformly quasi-coercive on
  $\Sob^1(\R)$. Moreover,
  $\widetilde u\in\Leb^2(0,1;\Sob^1(\R))$, and the displayed equation itself
  gives
  $\partial_t\widetilde u\in\Leb^2(0,1;\Sob^{-1}(\R))$ because
  $\widetilde f\in\Leb^2(0,1;\Leb^2(\R))$ and
  $\fra\partial_x\widetilde u\in\Leb^2(0,1;\Leb^2(\R))$.
  Thus the uniqueness assertion in Lions' variational theorem
  \cite{Lions1961}
  shows that $\widetilde u$ is the unique variational solution with initial value
  zero and forcing $\widetilde f$. If
  $\partial_t\widetilde u\in\Leb^2(0,1;\Leb^2(\R))$, its restriction to $(0,\pi)$
  would put $\partial_tu$ in $\Leb^2(0,1;H)$, contradicting
  Lemma~\ref{lem:solution-regularity}. This completes the proof.
\end{proof}

\begin{proposition}[Counterexamples in full space]
  \label{prop:tensor}
  Let $d\ge2$. Then the isotropic coefficient matrix
  \[
    \frB_d(t,x)=\fra(t,x_1)\Id_d
    \in\C^{0,1/2}\bigl([0,1];
    \Leb^\infty(\R^d;\R^{d\times d})\bigr)
  \]
  is real symmetric, bounded, and uniformly elliptic. Moreover, there exist
  a real-valued $F\in\Leb^2(0,1;\Leb^2(\R^d))$ and a unique Lions variational
  solution $U$ with $U(0)=0$ such that
  \[
    \partial_tU-\operatorname{div}(\frB_d\nabla U)=F,
    \qquad
    \partial_tU\notin\Leb^2(0,1;\Leb^2(\R^d)).
  \]
\end{proposition}

\begin{proof}
  Fix $d\ge2$. Let $\chi\in\C_c^\infty(\R^{d-1})$ be nonzero and real-valued,
  and denote by $\Delta'$ the Laplacian in the $x'$-variables. On $\R^d$,
  we set
  \[
    U=\widetilde u\otimes\chi,
    \qquad \frB_d(t,x)=\fra(t,x_1)\Id_d,
    \qquad
    F=\widetilde f\otimes\chi-\fra\widetilde u\otimes\Delta'\chi.
  \]
  Tensoring with $\chi$ defines bounded maps on the relevant Sobolev and
  Bochner spaces, and a straightforward calculation gives
  \[
    \partial_tU-\operatorname{div}(\frB_d\nabla U)=F
  \]
  in distributions.
  Here $F\in\Leb^2(0,1;\Leb^2(\R^d))$ and $U(0)=0$. Hence Lions'
  uniqueness theorem identifies $U$ as the variational solution.

  Let us now explain the failure of maximal regularity. For the nonzero
  real-valued function $\chi\in\Leb^2(\R^{d-1})$, define
  \[
    P_\chi\Psi(x_1)
    :=
    \norm{\chi}_{\Leb^2(\R^{d-1})}^{-2}
    \int_{\R^{d-1}}\Psi(x_1,x')\chi(x')\dd x'.
  \]
  By Cauchy--Schwarz and Fubini,
  \[
    P_\chi:
    \Leb^2(\R^d)\longrightarrow\Leb^2(\R)
  \]
  is bounded. For distributionally differentiable functions of time,
  bounded linearity gives
  \[
    \partial_t(P_\chi U)=P_\chi(\partial_tU)
  \]
  in the distributional sense. Since $P_\chi U=\widetilde u$, the assumption
  \[
    \partial_tU\in\Leb^2\bigl(0,1;\Leb^2(\R^d)\bigr)
  \]
  would imply
  \[
    \partial_t\widetilde u=P_\chi(\partial_tU)
    \in\Leb^2\bigl(0,1;\Leb^2(\R)\bigr),
  \]
  contradicting Proposition~\ref{prop:line}. This completes the proof.
\end{proof}

To treat an arbitrary bounded domain, we now localise the compactly supported
full-space solution by parabolic rescaling.

\begin{proof}[Proof of Corollary~\ref{cor:bounded-domains}]
  We keep the parameter $\eps>0$ in the construction free for the moment.
  If $d=1$, set
  \[
    W=\widetilde u,
    \qquad
    G=\widetilde f,
    \qquad
    \frB_d(s,z)=\fra(s,z)\Id_1.
  \]
  If $d\ge2$, let $Q'=(-1,1)^{d-1}$, choose a nonzero real
  $\chi\in\C_c^\infty(Q')$, and, writing $z=(z_1,z')$, set
  \[
    W(s,z)=\widetilde u(s,z_1)\chi(z'),
    \qquad
    \frB_d(s,z)=\fra(s,z_1)\Id_d,
  \]
  and
  \[
    G(s,z)=\widetilde f(s,z_1)\chi(z')
    -\fra(s,z_1)\widetilde u(s,z_1)\Delta'\chi(z').
  \]
  By Lemma~\ref{lem:solution-regularity}, Proposition~\ref{prop:forcing}, and the
  extension, product, and projection arguments from the proofs of
  Propositions~\ref{prop:line} and~\ref{prop:tensor}, in either case
  \[
    \partial_sW-\operatorname{div}_z(\frB_d\nabla_zW)=G
    \quad\text{in }\mathcal D'((0,1)\times\R^d),
  \]
  with
  \[
    W\in\C([0,1];\Sob^1(\R^d))
    \cap W^{1,1}(0,1;\Leb^2(\R^d)),
    \qquad
    G\in\C([0,1];\Leb^2(\R^d)),
  \]
  and
  \[
    \partial_sW\notin\Leb^2(0,1;\Leb^2(\R^d)).
  \]

  By the support properties of $\rho$ and the time blocks, $W$ and $G$
  vanish on neighbourhoods of both $s=0$ and $s=1$. We extend them by zero
  in the time variable and retain the already defined coefficient $\frB_d$ on
  $\R\times\R^d$. The preceding equation then holds on $\R\times\R^d$.
  Moreover, if
  \[
    K=
    \begin{cases}
      [0,\pi],                    & d=1,   \\
      [0,\pi]\times\overline{Q'}, & d\ge2,
    \end{cases}
  \]
  then
  \[
    \supp W(s,\cdot)\cup\supp G(s,\cdot)\subset K
    \qquad\text{for every }s\in\R.
  \]

  Since $\Omega$ is nonempty and open, there exist $x_0\in\R^d$ and
  $0<r\le1$ such that
  \[
    x_0+rK\Subset\Omega.
  \]
  For $(t,x)\in[0,1]\times\Omega$, set
  \[
    s=r^{-2}t,
    \qquad
    z=r^{-1}(x-x_0),
  \]
  where $z_1$ denotes the first coordinate of $z$, with $z_1=z$ when
  $d=1$. Define
  \[
    \begin{aligned}
      \frb(t,x) & =\fra(s,z_1),
                & \frB(t,x)     & =\frb(t,x)\Id_d=\frB_d(s,z), \\
      U(t,x)    & =W(s,z),
                & F(t,x)        & =r^{-2}G(s,z).
    \end{aligned}
  \]
  The parabolic change of variables gives
  \[
    U'-\operatorname{div}_x(\frB\nabla_xU)=F
    \quad\text{in }\mathcal D'((0,1)\times\Omega).
  \]
  Since $U(t,\cdot)$ and $F(t,\cdot)$ are supported in the fixed compact
  set $x_0+rK\Subset\Omega$, the regularity of $W$ and $G$ gives
  \[
    U\in\C([0,1];V_\Omega)\cap W^{1,1}(0,1;H_\Omega),
    \qquad
    F\in\C([0,1];H_\Omega).
  \]
  The zero extensions in the time variable also satisfy
  $W\in\C(\R;\Sob^1(\R^d))$ and
  \[
    \partial_sW
    =G+\operatorname{div}_z(\frB_d\nabla_zW)
    \in\C(\R;\Sob^{-1}(\R^d)).
  \]
  The spatial dilation and restriction maps are bounded in these spaces, so
  $U\in\C^1([0,1];V_\Omega')$. All terms in the equation are therefore
  continuous with values in $V_\Omega'$, and the distributional identity
  holds for every $t\in[0,1]$. Moreover, $U(0)=0$. Uniform ellipticity and
  Lions' theorem identify $U$ as the unique variational solution with initial
  value zero and forcing $F$.

  Proposition~\ref{prop:coefficient} and the time rescaling give
  \[
    \norm{\frB-\Id_d}_{\C^{0,1/2}
      ([0,1];\Leb^\infty(\Omega;\R^{d\times d}))}
    \le Cr^{-1}\eps.
  \]
  We choose $\eps>0$ sufficiently small that $\frB$ is uniformly elliptic and
  the right-hand side is less than $\delta$. For $\sigma,\tau\in[0,1]$,
  the inclusion $x_0+rK\Subset\Omega$ gives
  \[
    \norm{\frB(r^2\sigma,\cdot)-\frB(r^2\tau,\cdot)}
    _{\Leb^\infty(\Omega;\R^{d\times d})}
    \ge
    \norm{\fra(\sigma,\cdot)-\fra(\tau,\cdot)}
    _{\Leb^\infty(0,\pi)}.
  \]
  Proposition~\ref{prop:holder-sharp} now proves the asserted failure of
  higher H\"older regularity.

  Finally, another change of variables gives
  \[
    \int_0^1\norm{U'(t)}_{H_\Omega}^2\dd t
    =r^{d-2}\int_0^{r^{-2}}
    \norm{\partial_sW(s)}_{\Leb^2(\R^d)}^2\dd s
    =r^{d-2}\int_0^1
    \norm{\partial_sW(s)}_{\Leb^2(\R^d)}^2\dd s
    =\infty.
  \]
  Since $F\in\Leb^2(0,1;H_\Omega)$, the equation also rules out square
  integrability of the elliptic part. This completes the proof.
\end{proof}

\section{Comparison with endpoint hypotheses in the literature}
\label{sec:endpoint-exclusions}

In this section, we compare the construction with the known sufficient
hypotheses at the endpoint that were recalled in the introduction. The exclusions follow directly from
the corresponding theorems: if one of their sufficient hypotheses held, it
would contradict Theorem~\ref{thm:main}.

\begin{corollary}[Failure of the known sufficient hypotheses at the endpoint]
  Let $X=\Leb^\infty(0,\pi)$. Then the following assertions hold.
  \begin{enumerate}[label=\textup{(\roman*)}]
    \item the associated form does not have bounded variation in the sense of
          Dier \cite[Theorem~4.1]{Dier2015}, and
          $\fra\notin BV([0,1];X)$;
    \item the form path does not satisfy the Dini hypothesis of
          Haak--Ouhabaz \cite[Theorem~1.2]{HaakOuhabaz2015}: there is no
          non-decreasing function $\omega:[0,1]\to[0,\infty)$ such that
          \[
            \abs{\langle(\mathcal A(t)-\mathcal A(s))v,w\rangle_{V',V}}
            \le\omega(\abs{t-s})\norm{v}_V\norm{w}_V
          \]
          for all $s,t\in[0,1]$ and $v,w\in V$, and
          \[
            \int_0^1\frac{\omega(r)}{r^{3/2}}\dd r<\infty;
          \]
          in particular, no such modulus controls
          $\norm{\fra(t)-\fra(s)}_X$;
    \item neither $\mathcal A:[0,1]\to\mathcal L(V,V')$ nor
          $\fra:[0,1]\to X$ is piecewise $\Sob^{1/2}$ on any finite partition,
          as in
          \cite[Theorem~2.2]{AchacheOuhabaz2019};
    \item there is no $M<\infty$ for which
          \begin{equation}\label{eq:auscher-egert-condition}
            \sup_{\substack{I\subset[0,1]\\ \abs{I}>0}}\frac1{\abs{I}}
            \iint_{I\times I}
            \frac{\abs{\fra(t,x)-\fra(s,x)}^2}{\abs{t-s}^2}\dd s\dd t
            \le M
            \qquad\text{for a.e. }x\in(0,\pi);
          \end{equation}
          in other words, $\fra$ fails the scale-invariant square condition of
          Auscher--Egert
          \cite[Theorem~2, condition~(8)]{AuscherEgert2016}.
  \end{enumerate}
\end{corollary}

\begin{proof}
  We first consider bounded variation. If the associated symmetric coercive
  form family had bounded variation in Dier's sense, his theorem
  \cite[Theorem~4.1]{Dier2015} would prove maximal
  $\Leb^2$-regularity in $H$, contrary to Theorem~\ref{thm:main}. Moreover,
  \[
    \norm{\mathcal A(t)-\mathcal A(s)}_{\mathcal L(V,V')}
    \le\norm{\fra(t)-\fra(s)}_X.
  \]
  Indeed, if $\fra\in BV([0,1];X)$, let
  \[
    G(t):=\operatorname{Var}_X(\fra;[0,t]),
  \]
  where $\operatorname{Var}_X(\fra;[a,b])$ denotes the total variation
  of the $X$-valued path $\fra$ on $[a,b]$. Then $G$ is increasing and
  \[
    \abs{\langle(\mathcal A(t)-\mathcal A(s))v,w\rangle_{V',V}}
    \le\bigl(G(t)-G(s)\bigr)\norm{v}_V\norm{w}_V,
    \qquad 0\le s\le t\le1.
  \]
  Thus bounded variation of the coefficient path would imply Dier's
  bounded-variation estimate for the forms and is impossible.

  If the form path admitted a modulus $\omega$ as in part~\textup{(ii)},
  then \cite[Theorem~1.2]{HaakOuhabaz2015} would give maximal
  $\Leb^2$-regularity for zero initial data, again contradicting
  Theorem~\ref{thm:main}. The coefficient estimate
  \[
    \abs{\langle(\mathcal A(t)-\mathcal A(s))v,w\rangle_{V',V}}
    \le\norm{\fra(t)-\fra(s)}_X\norm{v}_V\norm{w}_V
  \]
  shows that a Dini modulus for $\fra$ would give the excluded form modulus.

  Let us next verify the two side conditions in
  \cite[Theorem~2.2]{AchacheOuhabaz2019}. The $H$-realisations $A_H(t)$ are
  positive self-adjoint, and the
  spectral theorem and uniform ellipticity give
  \[
    \norm{A_H(t)^{1/2}v}_H^2
    =\langle\mathcal A(t)v,v\rangle_{V',V}
    =\int_0^\pi\fra(t,x)\abs{\partial_xv(x)}^2\dd x
    \simeq\norm{v}_V^2
  \]
  uniformly in $t$; hence the uniform Kato square-root property holds. Also,
  Proposition~\ref{prop:coefficient} and the preceding operator estimate imply
  for every interval $I\subset[0,1]$ that
  \[
    \sup_{t\in I}\int_I
    \frac{\norm{\mathcal A(t)-\mathcal A(s)}_{\mathcal L(V,V')}^2}{\abs{t-s}}\dd s
    \le C\eps^2\abs{I}.
  \]
  A sufficiently fine finite partition therefore gives the local smallness
  required in condition~(2.1) of that paper. If $\mathcal A$ were piecewise
  $\Sob^{1/2}$, one could refine its partition by this smallness partition and
  apply the cited theorem, contradicting Theorem~\ref{thm:main}.
  Piecewise $\Sob^{1/2}$ regularity of $\fra$ would imply the same regularity of
  $\mathcal A$ by the displayed estimate, so it fails, too.

  Finally, under boundedness and ellipticity, condition
  \eqref{eq:auscher-egert-condition} is precisely the sufficient coefficient
  hypothesis in \cite[Theorem~2, condition~(8)]{AuscherEgert2016} for maximal
  $\Leb^2$-regularity of the Dirichlet divergence-form operator on the open
  set $(0,\pi)$.
\end{proof}

$ $
\vfill

\end{document}